\documentclass[12pt, a4paper]{article}
\usepackage{amssymb}
\usepackage[all,arc]{xy}
\usepackage{mathrsfs}
\usepackage{amsfonts}
\usepackage{color}
\usepackage{amsmath,amsthm,amsfonts,amssymb,amscd}

\allowdisplaybreaks[4]
\newtheorem{defn}{Definition}[section]
\newtheorem{them}[defn]{Theorem}
\newtheorem{lem}[defn]{Lemma}

\newtheorem{cor}[defn]{Corollary}

\newtheorem{prop}[defn]{Proposition}

\newtheorem{con}[defn]{Conjecture}

\numberwithin{equation}{section}

\begin{document}
\title{Distance and distance signless Laplacian spread of connected   graphs
\footnote{L. You's   work was supported by the National Natural Science Foundation of China (Grant No. 11571123)
and the Guangdong Provincial Natural Science Foundation(Grant No. 2015A030313377),
and G. Yu's work was supported by the NSF of China (Grant No. 11271315).}}

 \author{Lihua You$^{a,}$\footnote{{\it{Email address:\;}}ylhua@scnu.edu.cn}
\qquad Liyong Ren$^{a,}$\footnote{{\it{Email address:\;}}275764430@qq.com. }
 \qquad Guanglong Yu$^{b,}$\footnote{{\it{Corresponding author:\;}}yglong01@163.com.}}
\vskip.2cm
\date{{\small
$^{a}$ School of Mathematical Sciences, South China Normal University, Guangzhou, 510631, P.R. China\\
$^{b,}$ Department of Mathematics, Yancheng Teachers University, Yancheng, 224002, Jiangsu, P.R. China\\
}}
\maketitle
\noindent {\bf Abstract }
For a connected graph $G$ on $n$ vertices, recall that the distance signless Laplacian matrix of $G$ is defined to be $\mathcal{Q}(G)=Tr(G)+\mathcal{D}(G)$, where $\mathcal{D}(G)$ is the distance matrix, $Tr(G)=diag(D_1, D_2, \ldots, D_n)$ and $D_{i}$ is the row sum of $\mathcal{D}(G)$ corresponding to vertex $v_{i}$.
Denote by $\rho^{\mathcal{D}}(G),$ $\rho_{min}^{\mathcal{D}}(G)$ the largest eigenvalue and the least eigenvalue of $\mathcal{D}(G)$, respectively. And denote by $q^{\mathcal{D}}(G)$, $q_{min}^{\mathcal{D}}(G)$ the largest eigenvalue and the least eigenvalue of $\mathcal{Q}(G)$, respectively. The distance spread of a graph $G$ is defined as $S_{\mathcal{D}}(G)=\rho^{\mathcal{D}}(G)- \rho_{min}^{\mathcal{D}}(G)$, and
the distance signless Laplacian  spread of a graph $G$ is defined as $S_{\mathcal{Q}}(G)=q^{\mathcal{D}}(G)-q_{min}^{\mathcal{D}}(G)$.
In this paper, we point out an error in the result of Theorem 2.4 in  ``Distance spectral spread of a graph" [G.L. Yu, et al,
Discrete Applied Mathematics. 160 (2012) 2474--2478] and rectify it.
As well, we obtain some lower bounds on ddistance signless Laplacian spread of a graph.


\section{Introduction}
\hskip 0.6cm
Throughout this article, we assume that $G$ is a simple, connected and undirected graph on $n$ vertices. Let $G=(V(G), E(G))$ be a graph with vertex set $V(G)=\{v_1,v_2,\ldots,v_n\}$ and edge set $E(G)$.
 We denote by $deg(v_i)$ (simply, $d_i$) the degree of vertex $v_i$, and for $u, v\in V$, we denote by $d_G(u,v)$ the distance between $u$ and $v$ in $G$.
Recall that the $distance$ $matrix$ is $\mathcal{D}(G)=(d_{ij})$ where $d_{ij}=d_G(v_i,v_j)$.
For any $v_{i}\in V(G)$, the $transmission$ of vertex $v_{i}$, denoted by $Tr_G(v_i)$ or $D_{i}$, is defined to be $\sum\limits_{v_{j}\in V(G), j\neq i}d_G(v_i, v_{j})$, which is equal to the row sum of $\mathcal{D}(G)$ corresponding to vertex $v_i$. Sometimes, $D_{i}$ is called the $distance$ $degree$.
Let $Tr(G)=diag(D_1, D_2, \ldots, D_n)$ be the diagonal matrix of vertex transmissions of $G$.
  The $distance$ $signless$ $Laplacian$ $matrix$ of $G$ is
   defined as
 $ \mathcal{Q}(G)=Tr(G)+\mathcal{D}(G)$ (see \cite{2013A}).

For  a nonnegative real symmetric matrix $M$, we denote by $P_{M}(\lambda)=|\lambda I- M |$ its the characteristic polynomial. Its largest eigenvalue is called the spectral radius of $M$. For a graph $G$, the $spectral$ $radius$ of $\mathcal{D}(G)$  and $\mathcal{Q}(G)$,  denoted by $\rho^{\mathcal{D}}(G)$ and
$q^{\mathcal{D}}(G)$, are also called the $distance$ $spectral$ $radius$  and  the $distance$ $signless$ $Laplacian$ $spectral$ $radius$, respectively. Denote by $\rho_{min}^{\mathcal{D}}(G)$ and $q_{min}^{\mathcal{D}}(G)$ the least eigenvalue of $\mathcal{D}(G)$ and the least eigenvalue of $\mathcal{Q}(G)$, respectively.  The $distance$ $spread$ of graph $G$ is defined as $S_{\mathcal{D}}(G)=\rho^{\mathcal{D}}(G)- \rho^{\mathcal{D}}_{min}(G)$, and the $distance$ $signless$ $Laplacian$ $spread$ of  graph $G$ is defined as
$S_{\mathcal{Q}}(G)=q^{\mathcal{D}}(G)- q^{\mathcal{D}}_{min}(G)$. Without ambiguity, $S_{\mathcal{D}}(G)$ and $S_{\mathcal{Q}}(G)$ are shortened as $S_{\mathcal{D}}$ and $S_{\mathcal{Q}}$ sometimes.

From \cite{1985, 1956}, we know that the spread of a matrix is a very interesting topic. As a result, in algebraic graph theory, the spread of some matrices of a graph also becomes interesting (see \cite{2001GHK}, \cite{2010LL}). Because the research of the eigenvalues of the distance matrix of a graph is of great significance for both algebraic graph theory and practical applications, the problem concerning the distance spectrum of a graph has been studied extensively recently (see \cite{1971}, \cite{1991}, \cite{1998}, \cite{2014A}). These cause the interests of the researchers on the problem about the distance spectral spread of a graph (\cite{2012}, \cite{2015}).
Motivated by these, in this paper, we proceed to consider the distance and distance signless Laplacian spread of a graph.

In Section 3,  we point out an error in the result of  Theorem 2.4 in  ``Distance spectral spread of a graph" [G.L. Yu, etc,
Discrete Applied Mathematics. 160 (2012) 2474--2478] and rectify it. In Section 4, some lower bounds on distance signless Laplacian   spread of a graph are obtained.

\section{Some preliminaries}
\hskip 0.6cm In this section, we introduce some definitions, notations and working lemmas.

Let $I_p$ be the $p\times p$ identity matrix and $J_{p,q}$ be the $p\times q$ matrix in which every entry is $1$, or simply $J_p$ if $p=q$. For a matrix $M$, its spectrum $\sigma(M)$ is the multiset of its eigenvalues.

\begin{defn}\label{defn21}
Let $M$ be a real matrix of order $n$ described in the following block form
\begin{equation}\label{eq21}
M = \left(\begin{array}{ccc}
M_{11} & \cdots & M_{1t}\\
\vdots &  \ddots      &\vdots \\
 M_{t1}& \cdots & M_{tt}\\
\end{array}\right),
\end{equation}

\noindent where the diagonal blocks $M_{ii}$ are $n_i\times n_i$ matrices for any $i\in\{1,2,\ldots, t\}$ and $n=n_1+\ldots+n_t$.
For any $i,j\in\{1,2,\ldots, t\}$, let $b_{ij}$ denote the average row sum of $M_{ij}$, i.e. $b_{ij}$ is the sum of all entries in $M_{ij}$ divided by the number of rows. Then $B(M) = (b_{ij})$ (simply by $B$) is called the quotient matrix of $M$. If in addition for each pair $i, j$, $M_{ij}$ has constant row sum, then $B(M)$ is called the equitable quotient matrix of $M$.
\end{defn}

Consider two sequences of real numbers: $\lambda_{1}\geq \lambda_{2} \geq ... \geq\lambda_{n}$, and $\mu_{1}\geq \mu_{2}\geq ...\geq \mu_{m}$  with $m<n$. The second sequence is said to interlace the first one whenever
$\lambda_{i}\geq \mu_{i}\geq \lambda_{n-m+i}$ for $i=1,2,...,m$.




\begin{lem}{\rm(\cite{1995H})}\label{lem22}
Let $M$ be a symmetric matrix and  have the block form as (\ref{eq21}),
$B$ be the quotient matrix of $M$. Then the eigenvalues of $B$ interlace the eigenvalues of $M$.
\end{lem}


\begin{lem}\label{lem24}{\rm(\cite{2015MYL})}
Let $M$ be defined as (\ref{eq21}), and for any $i,j \in\{ 1,2\ldots,t\}$,
 $M_{ii} = l_iJ_{n_i} + p_iI_{n_i},$ $M_{ij} = s_{ij}J_{n_i,n_j}$ for $i\not= j$, where $l_i, p_i, s_{ij}$ are real numbers,
  $B=B(M)$ be the  quotient matrix of $M$.  Then
\begin{equation}\label{eq22}
\sigma(M)=\sigma(B)\cup \{p_i^{[n_i-1]} \mid i = 1,2\ldots,t\},
\end{equation}
 where $\lambda^{[t]}$ means that $\lambda$ is an eigenvalue with multiplicity $t$.
\end{lem}

By Lemma \ref{lem24}, we can obtain the distance (signless Laplacian) spectrum of $K_{a,b}$ as follows  immediately, where $n=a+b$.

\begin{equation}\label{eq23}
\sigma(\mathcal{D}(K_{a,b}))=\left\{(-2)^{[n-2]},  n-2\pm\sqrt{n^{2}-3ab} \right\},
\end{equation}
and
\begin{equation}\label{eq24}
\sigma(\mathcal{Q}(K_{a,b}))=\left\{(2n-a-4)^{[b-1]},(2n-b-4)^{[a-1]}, \frac{5n-8 \pm \sqrt{9n^{2}-32ab}}{2}\ \right\}.
\end{equation}

\begin{lem}{\rm(\cite{1995})}\label{lem25}
 Let $H_{n}$ denote the set of all $n\times n$ Hermitian matrices, $A \in  H_{n}$ with eigenvalues
$\lambda_{1} \geq \lambda_{2} \geq ...\geq \lambda_{n}$,  and $B$ be a $m\times m$ principal matrix of $A$
 with eigenvalues $\mu_{1} \geq \mu_{2} \geq ... \geq \mu_{m}$.
 Then $\lambda_{i}\geq \mu_{i} \geq \lambda_{n-m+i}$ for $i = 1, 2,..., m$.
\end{lem}

\section{Results on $S_{\mathcal{D}}$ for  a bipartite graph}

\hskip 0.6cm
In \cite{2012}, the authors obtained a lower bound for $S_{\mathcal{D}}(G)$ with the maximum degree $\Delta$ of $G$,
but it is found that the result is incorrect when $\Delta\leq |V(G)|-2$. In this section, we rectify it.

Let $G=(V, E)$ be a graph. For $v_i, v_j\in V$, if $v_i$ is adjacent to $v_j$,  we denote it by $v_i\sim v_j$ (simply, $i\sim j$). We let $t_{v}=\frac{\sum\limits_{v_{i}\sim v}{D_{i}}}{d_{v}}$ be the $average$ $distance$ $degree$ of $v$ (\cite{2012}).


\begin{prop}{\rm(\cite{2012}, Theorem 2.4)}\label{prop301}
Let $G$ be a simple connected bipartite graph on $n$ vertices with $S=\sum\limits_{i=1}^{n}D_{i}$
and maximum degree $\Delta$. Suppose $deg(v_{1})=deg(v_{2})=\cdots= deg(v_{k})=\Delta$.
Then

\noindent {\rm (i)}  if $\Delta \leq n-2,$ we have
\begin{equation}\label{eq301}
S_{\mathcal{D}}(G)\geq \max\limits_{1\leq i\leq k}{\frac{\sqrt{a_{i}^{2}-4 b_{i}(\Delta+1)(n-\Delta-1)}}{(\Delta+1)(n-\Delta-1)}},
\end{equation}
where $ a_{i}=2(n-t_{v_{i}}-1)\Delta^{2}
+(S-2t_{v_{i}}-2)\Delta+S$ and  $b_{i}=\Delta^{2}(2S-t_{v_{i}}^{2}-2t_{v_{i}}-1).$

\noindent {\rm (ii)}  if $\Delta=n-1$, we have
\begin{equation*}
  S_{\mathcal{D}}(G)= \begin{cases}
              0, & \mbox{if } n=1; \\
              2, & \mbox{if } n=2; \\
              n+\sqrt{n^{2}-3n+3}, & \mbox{if }  n\geq 3.
            \end{cases}
\end{equation*}
\end{prop}

Let $ N(v_{i})=\{v_{i_{1}}, v_{i_{2}}, \ldots, v_{i_{\Delta}}\}$ be the neighbors set of $v_{i}$,
 and $N[v_i]=N(v_i)\cup\{v_i\}$. In the proof of (\ref{eq301}),
the authors partition $V(G)$ into two parts $N[v_{i}]$ and $V(G)\setminus N[v_{i}]$ for some $1\leq i \leq k$.
  Corresponding to this partition, $\mathcal{D}(G)$ can be written as

  \begin{equation}\label{eq302}
\mathcal{D}(G)
= \begin{pmatrix}
            \begin{matrix}   0 & 1 & 1 &  \dots & 1 \\ 1 & 0 & 2 & \dots & 2 \\ 1 & 2 & 0 & \dots & 2 \\ \dots & \dots & \dots & \dots & \dots  \\ 1 &  2 & 2 & \dots & 0 \end{matrix}
            & \text{\Large * } \\
              \text{\Large * }
             & \text{\Large * } \\
          \end{pmatrix}.\end{equation}

Then the author presented  the quotient matrix of $\mathcal{D}(G)$ as:
\begin{equation}\label{eq303}
B_1= \begin{pmatrix}
   \frac{2\Delta ^{2}}{\Delta +1} & \frac{t_{v_{i}}\Delta +\Delta-2\Delta^{2}}{\Delta +1} \\
    \frac{t_{v_{i}}\Delta +\Delta-2\Delta^{2}}{n-\Delta -1 } & \frac{S-2 t_{v_{i}} \Delta +2\Delta(\Delta-1)}{n-\Delta -1}
     \end{pmatrix}.\end{equation}

The following example shows that (\ref{eq303}) is false.

\setlength{\unitlength} {4mm}
\begin{center}
\begin{picture}(15,10)
\put(3,5){\circle* {0.1}}   \put(5,5){\circle* {0.1}}     \put(10,5){\circle* {0.1}}     \put(3,10){\circle* {0.1}}     \put(5,10){\circle* {0.1}}
\put(10,10){\circle* {0.1}}     \put(13,10){\circle* {0.1}}
\put(2,4){$v_{7}$}          \put(5,4){$v_{3}$}         \put(10.3,4){$v_{6}$}       \put(2,9){$v_{4}$}          \put(4,9){$v_{1}$}
\put(8.5,9){$v_{2}$}          \put(13.3,9){$v_{5}$}
\put(3,5){\line(2,0){2}}    \put(5,5){\line(0,5){5}}      \put(5,5){\line(5,0){5}}      \put(5,10){\line(-2,0){2}}       \put(5,10){\line(5,0){5}}      \put(10,5){\line(0,5){5}}    \put(10,10){\line(3,0){3}}
\end{picture}
\end{center}
 \vskip-1.7cm\hskip6.5cm  Fig. 3.1. \hskip.1cm  $G_1$
 \vskip0.2cm
For the above graph (see Fig. 3.1), it is clear that $\Delta=3, t_{v_1}=\frac{34}{3}, S=84$ and
\begin{equation}\label{eq304}
\mathcal{D}(G)=(d_{ij})_{7\times 7}=\left(
                                      \begin{array}{cccc|ccc}
             0 & 1 & 1 & 1 & 2 & 2 & 2 \\
             1 & 0 & 2 & 2 & 1 & 1 & 3 \\
             1 & 2 & 0 & 2 & 3 & 1 & 1 \\
             1 & 2 & 2 & 0 & 3 & 3 & 3 \\\hline
             2 & 1 & 3 & 3 & 0 & 2 & 4 \\
             2 & 1 & 1 & 3 & 2 & 0 & 2 \\
             2 & 3 & 1 & 3 & 4 & 2& 0
                                      \end{array}
                                    \right),\end{equation}
 and by (\ref{eq303}), we have the quotient matrix
$B_1=\begin{pmatrix}
   \frac{18}{4} & \frac{19}{4} \\
    \frac{19}{3 } & \frac{28}{3}
\end{pmatrix}.$
In fact  the quotient matrix  can be computed by  the definition of  the quotient matrix and (\ref{eq304}) immediately  as
$B_{2}= \begin{pmatrix}
     \frac{18}{4} & \frac{25}{4} \\
    \frac{25}{3 } & \frac{16}{3}
\end{pmatrix}\not=B_1,$ it is a contradiction.

Noticing the error in (\ref{eq303}), with the similar technique,
we rectify (\ref{eq301}) as follows.

\begin{them}\label{thm302}
 Let $G$ be a simple connected bipartite graph on $n$ vertices, $\Delta$ be the  maximum degree of $G$,
 $S=\sum\limits_{i=1}^{n}D_{i}$.
 Suppose that  $deg(v_1)=deg(v_2)=\ldots=deg(v_k)=\Delta \leq n-2$ 
 for some $k\ (1\leq k\leq n)$. Then
\begin{equation}\label{eq305}
S_{\mathcal{D}}(G)\geq \max\limits_{1\leq i\leq k}{\frac{\sqrt{a_{i}^{2}+4 b_{i}(1+\Delta)(n-\Delta-1)}}{(1+\Delta)(n-\Delta-1)}},
\end{equation}
\noindent where $a_{i}=(\Delta+1)(S-2D_{i}-2t_{v_{i}}\Delta)+2n\Delta^{2}$ and $b_{i}=D_{i}^{2}-2S\Delta^{2}+2D_{i}t_{v_{i}}\Delta+t_{v_{i}}^{2}\Delta^{2}.$
\end{them}
\begin{proof}
$V(G)$ is partitioned  into two parts which are  $N[v_{i}]$ and $V(G)\setminus N[v_{i}]$ for some $1\leq i \leq k$.
  Corresponding to this partition, $\mathcal{D}(G)$ is written as (\ref{eq302})
and the quotient matrix $B$ of $\mathcal{D}(G)$ is presented  as follow:
\begin{equation*}
B= \begin{pmatrix}
   \frac{2\Delta ^{2}}{\Delta +1} & \frac{t_{v_{i}}\Delta +D_{i}-2\Delta^{2}}{\Delta +1} \\
    \frac{t_{v_{i}}\Delta +D_{i}-2\Delta^{2}}{n-\Delta -1 } & \frac{S-2 t_{v_{i}} \Delta +2\Delta^{2}-2D_{i}}{n-\Delta -1}
\end{pmatrix}.\end{equation*}

Then
\begin{equation*}  \begin{array}{l}
P_{B}(\lambda)=|\lambda I- B |
=\lambda^{2}-\frac{(\Delta+1)(S-2D_{i}-2t_{v_{i}}\Delta)+2n\Delta^{2}}{(1+\Delta)(n-\Delta-1)}\lambda -\frac{D_{i}^{2}-2S\Delta^{2}+2D_{i}t_{v_{i}}\Delta+t_{v_{i}}^{2}\Delta^{2}}{(1+\Delta)(n-\Delta-1)}
 \end{array}.
 \end{equation*}

Let $P_{B}(\lambda)=0$. It follows that
$$\lambda_{1,2}=\frac{a_{i}\pm\sqrt{a_{i}^{2}+4b_{i}(1+\Delta)(n-\Delta-1)}}{2(1+\Delta)(n-\Delta-1)},$$
where $a_{i}=(\Delta+1)(S-2D_{i}-2t_{v_{i}}\Delta)+2n\Delta^{2}$ and $b_{i}=D_{i}^{2}-2S\Delta^{2}+2D_{i}t_{v_{i}}\Delta+t_{v_{i}}^{2}\Delta^{2}.$
Using Lemma \ref{lem22} gets (\ref{eq305}).
\end{proof}

{\bf Remark 3.1}
\vskip 0.9cm \setlength{\unitlength} {4mm}
\begin{center}
\begin{picture}(15,10)
\put(5,5){\circle* {0.1}}     \put(10,5){\circle* {0.1}}    \put(5,8){\circle* {0.1}}     \put(10,8){\circle* {0.1}}
\put(5,11){\circle* {0.1}}     \put(10,11){\circle* {0.1}}     \put(13,11){\circle* {0.1}}   \put(13,8){\circle* {0.1}}
\put(13,5){\circle* {0.1}}
\put(4.5,4){$v_{5}$}                \put(9.5,4){$v_{6}$}             \put(4,10){$v_{3}$}           \put(9,10){$v_{2}$}          \put(4.1,7){$v_{4}$}
\put(9,7){$v_{1}$}                 \put(13.5,10.5){$v_{7}$}                  \put(13.5,7.5){$v_{8}$}              \put(13.5,4.5){$v_{9}$}
\put(5,8){\line(0,3){3}}     \put(5,8){\line(0,-3){3}}      \put(5,8){\line(5,0){5}}       \put(10,11){\line(-5,0){5}}      \put(10,11){\line(0,-3){3}}           \put(10,5){\line(0,3){3}}     \put(10,5){\line(-5,0){5}}    \put(10,8){\line(1,1){3}}       \put(10,8){\line(3,0){3}}       \put(10,8){\line(1,-1){3}}
\put(4.5,1){Fig. 3.2. \hskip.1cm  $G_2$}

\end{picture}
\end{center}

\begin{center}
\begin{tabular}{|c|c|c|}
  \hline
 graph   & Theorem \ref{thm302}                           & approximate   value  \\\hline
  $G_1$  & $S_{\mathcal{D}}(G)\geq 15.5960$              &$S_{\mathcal{D}}(G)\approx 17.6820$ \\\hline
  $G_2$  & $S_{\mathcal{D}}(G)\geq 19.0059$         &$S_{\mathcal{D}}(G)\approx 20.9674$ \\\hline
\end{tabular}
\\\vskip0.3cm  Table 3.1.
\end{center}

By computation with mathematica for graphs $G_1$ and $G_2$ (see Figs. 3.1, 3.2 and Table 3.1),
it seems that Theorems \ref{thm302} is useful to evaluate the distance spread of a bipartite graph.

From the proof of Theorem \ref{thm302} and by Lemma \ref{lem22}, we have the following corollary  immediately.
\begin{cor}\label{cor301}
Let $G$ be a simple connected bipartite graph on $n\geq 3$ vertices with maximum degree
$\Delta\leq n-2$. Suppose that
 $deg(v_1)=deg(v_2)=\ldots=deg(v_k)=\Delta$, $a_{i}, b_{i}$ are defined as Theorem \ref{thm302}
 for $1\leq i\leq k$. Then

\noindent{\rm (i) } $\displaystyle \rho^{\mathcal{D}}(G)\geq \max\limits_{1\leq i\leq k}\frac{a_{i}+\sqrt{a_{i}^{2}+4b_{i}(1+\Delta)(n-\Delta-1)}}{2(1+\Delta)(n-\Delta-1)}$;

\noindent{\rm (ii)} $\displaystyle \rho^{\mathcal{D}}_{min}(G)\leq \min\limits_{1\leq i\leq k}\frac{a_{i}+\sqrt{a_{i}^{2}+4b_{i}(1+\Delta)(n-\Delta-1)}}{2(1+\Delta)(n-\Delta-1)}$.
\end{cor}


\section{On $S_{\mathcal{Q}}$}
\hskip.6cm In this section, we show some bounds of $S_{\mathcal{Q}}$ for bipartite graphs and some bounds with some parameters.

\subsection{Bounds on $S_{\mathcal{Q}}$ for bipartite graphs}
\hskip.6cm For a graph $G$, $W(G)=\sum \limits_{1 \leq i< j\leq n} d_{ij}$ is called $Wiener$ $index $. Thus, $W(G)=\frac{1}{2}\sum \limits_{i=1}^{n} D_{i}$ and $S=2W(G)$.  Similar to the proof of Theorem \ref{thm302} and Corollary \ref{cor301}, we get the following  theorem and one corollary in term of $ Wiener$ $ index $.

\begin{them}\label{thm401}
  Let $G$ be a simple connected bipartite graph on $n\geq 3$ vertices with
 maximum degree $\Delta$ and Wiener index $W$.
 Suppose that $deg(v_1)=deg(v_2)=\ldots=deg(v_k)=\Delta$. Then

\noindent{\rm (i) }  if $\Delta \leq n-2,$ then
\begin{equation}\label{eq401}
S_{\mathcal{Q}}(G)\geq \max\limits_{1\leq i\leq k}{\frac{\sqrt{a_{i}^{2}+4b_{i}(1+\Delta)(n-\Delta-1)}}{(1+\Delta)(n-\Delta-1)}},
 \end{equation}
 where $a_{i}=4(W-D_{i}-t_{v_{i}}\Delta)(\Delta+1)+2n\Delta^{2}+nD_{i}+nt_{v_{i}}\Delta$ and $b_{i}=4D_{i}^{2}+8D_{i}t_{v_{i}}\Delta+4t_{v_{i}}^{2}\Delta^{2}-8W\Delta^{2}-4WD_{i}-4Wt_{v_{i}}\Delta.$

\noindent{\rm (ii) } if $\Delta=n-1$, then $S_{\mathcal{Q}}(G)=\sqrt{9n^{2}-32n+32}$.
\end{them}

{\bf Remark 4.1}
\vskip0.3cm
\begin{center}
\begin{tabular}{|c|c|c|}
  \hline
 graph   & Theorem \ref{thm401}                            & approximate   value  \\\hline
  $G_1$  & $S_{\mathcal{Q}}(G)\geq 15.6400$               &$S_{\mathcal{Q}}(G)\approx 18.6100$ \\\hline
  $G_2$  & $S_{\mathcal{Q}}(G)\geq 17.8520$        &$S_{\mathcal{Q}}(G)\approx 21.1870$ \\\hline
\end{tabular}
\\\vskip0.3cm  Table 4.1.
\end{center}

By computation with mathematica for graphs $G_1$ and $G_2$ (see Figs. 3.1, 3.2 and Table 4.1),
it seems that Theorem \ref{thm401} is useful to evaluate the signless Laplacian distance spread of a bipartite graph.

\vskip0.3cm
\begin{cor}
Let $G$ be a simple connected bipartite graph on $n\geq 3$ vertices,
$\Delta\leq n-2$ be maximum degree of $G$. Suppose that
 $deg(v_1)=deg(v_2)=\ldots=deg(v_k)=\Delta$
 for some $k$ $(1\leq k\leq n)$,
 $a_{i}, b_{i}$ are defined as Theorem \ref{thm401} for $1\leq i\leq k$.
 Then

\noindent {\rm (i) } $ \displaystyle q^{\mathcal{D}}(G)\geq \max\limits_{1\leq i\leq k}\frac{a_{i}+\sqrt{a_{i}^{2}+4b_{i}(1+\Delta)(n-\Delta-1)}}{2(1+\Delta)(n-\Delta-1)}$;

\noindent {\rm (ii)} $\displaystyle q^{\mathcal{D}}_{min}(G)\leq \min\limits_{1\leq i\leq k}\frac{a_{i}+\sqrt{a_{i}^{2}+4b_{i}(1+\Delta)(n-\Delta-1)}}{2(1+\Delta)(n-\Delta-1)}$.
\end{cor}

\begin{lem}\label{lem406}
Let $n\geq 4$ and $a$ be positive integers with $2a\leq n$. 
Then $S_{\mathcal{Q}}(K_{a,n-a})\geq S_{\mathcal{Q}}(K_{\lfloor\frac{n}{2} \rfloor,\lceil\frac{n}{2} \rceil })$
with equality if and only if $G\cong K_{\lfloor\frac{n}{2} \rfloor,\lceil\frac{n}{2} \rceil }.$
\end{lem}
\begin{proof}
By (\ref{eq24}), we have

$$\sigma(\mathcal{Q}(K_{a,n-a}))=\left\{(2n-a-4)^{[n-a-1]},(n+a-4)^{[a-1]}, \frac{5n-8 \pm \sqrt{9n^{2}-32a(n-a)}}{2}\ \right\}.$$

It is checked that
$$\frac{5n-8 +\sqrt{9n^{2}-32a(n-a)}}{2}>2n-a-4,\ \frac{5n-8-\sqrt{9n^{2}-32a(n-a)}}{2}>n+a-4.$$
Then $q^{D}(K_{a,n-a})=\frac{5n-8 + \sqrt{9n^{2}-32a(n-a)}}{2}$, and
 $$\displaystyle q^{D}_{min}(K_{a,n-a})=\left\{\begin{array}{cc}
                                n+a-4, & a>1 \\
                                \frac{5n-8-\sqrt{9n^{2}-32a(n-a)}}{2}, & a=1.
                              \end{array}\right.$$

 When $0<a\leq \frac{n}{2}$, it checked that $f(a)=\frac{5n-8 + \sqrt{9n^{2}-32a(n-a)}}{2}$ is a decreasing function with respect
 to $a$,
 $g(a)=n+a-4$ 
 is a increasing function with respect
 to $a$.
 Then we have $S_{\mathcal{Q}}(K_{2,n-2})>S_{\mathcal{Q}}(K_{3,n-3})>\ldots> S_{\mathcal{Q}}(K_{\lfloor\frac{n}{2} \rfloor,\lceil\frac{n}{2} \rceil })$.

 Noting that $n\geq 4$, by directly computation, we have

 $S_{\mathcal{Q}}(K_{1,n-1})-S_{\mathcal{Q}}(K_{2,n-2})$

$\displaystyle=\sqrt{9n^{2}-32n+32}-(\frac{5n-8+\sqrt{9n^{2}-64n+128}}{2}-(n-2))>0$.

Thus $S_{\mathcal{Q}}(K_{1,n-1})>S_{\mathcal{Q}}(K_{2,n-2})>S_{\mathcal{Q}}(K_{3,n-3})>\ldots> S_{\mathcal{Q}}(K_{\lfloor\frac{n}{2} \rfloor,\lceil\frac{n}{2} \rceil })$. This completes the proof.
\end{proof}

Let $G$ be a simple connected bipartite graph on $n$ vertices.
 If $n=4$, $G$ is isomorphic to one of  the following three graphs: (1) $K_{2,2}$, (2) $P_4$, (3) $S_4$;
  if $n=5$, $G$ is isomorphic to one of  the following five graphs: (4) $K_{2,3}$, (5) $G_5$, (6) $G_6$, (7) $P_5$, (8) $S_5$ (see Fig. 4.1).
\setlength{\unitlength} {4mm}
\begin{center}
\begin{picture}(38,10)
\put(1,1){\circle* {0.1}}   \put(3,1){\circle* {0.1}}     \put(5,1){\circle* {0.1}}     \put(2,3){\circle* {0.1}}     \put(4,3){\circle* {0.1}}
\put(0,0.7){$v_{3}$}          \put(2,0.7){$v_{4}$}         \put(5.3,0.7){$v_{5}$}       \put(1,3){$v_{1}$}          \put(4.3,3){$v_{2}$}
\put(1,1){\line(1,2){1}}    \put(1,1){\line(3,2){3}}      \put(3,1){\line(-1,2){1}}      \put(3,1){\line(1,2){1}}       \put(5,1){\line(-3,2){3}}      \put(5,1){\line(-1,2){1}}
\put(2,-1){ (4) $K_{2,3}$}

\put(8,1){\circle* {0.1}}   \put(10,1){\circle* {0.1}}     \put(12,1){\circle* {0.1}}     \put(9,3){\circle* {0.1}}      \put(11,3){\circle* {0.1}}
\put(7,0.7){$v_{3}$}          \put(9,0.7){$v_{4}$}         \put(12.3,0.7){$v_{5}$}       \put(8,3){$v_{1}$}             \put(11.3,3){$v_{2}$}
\put(8,1){\line(1,2){1}}    \put(8,1){\line(3,2){3}}       \put(10,1){\line(-1,2){1}}      \put(10,1){\line(1,2){1}}       \put(12,1){\line(-3,2){3}}     \put(9,-1){ (5) $H_{1}$}

\put(15,1){\circle* {0.1}}   \put(17,1){\circle* {0.1}}     \put(19,1){\circle* {0.1}}     \put(16,3){\circle* {0.1}}      \put(18,3){\circle* {0.1}}
\put(14,0.7){$v_{3}$}          \put(16,0.7){$v_{4}$}         \put(19.3,0.7){$v_{5}$}       \put(15,3){$v_{1}$}             \put(18.3,3){$v_{2}$} \put(15,1){\line(1,2){1}}    \put(15,1){\line(3,2){3}}       \put(17,1){\line(-1,2){1}}          \put(19,1){\line(-3,2){3}}
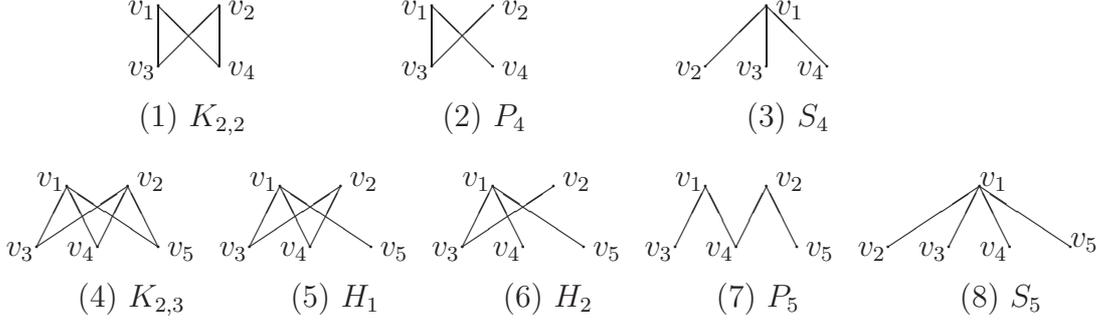
\put(16,-1){ (6) $H_{2}$} \put(13,-3){ Fig. 4.1. $K_{2,2}$-$S_{5}$}

\put(22,1){\circle* {0.1}}   \put(24,1){\circle* {0.1}}     \put(26,1){\circle* {0.1}}     \put(23,3){\circle* {0.1}}      \put(25,3){\circle* {0.1}}
\put(21,0.7){$v_{3}$}          \put(23,0.7){$v_{4}$}         \put(26.3,0.7){$v_{5}$}       \put(22,3){$v_{1}$}             \put(25.3,3){$v_{2}$}
\put(22,1){\line(1,2){1}}     \put(24,1){\line(-1,2){1}}      \put(24,1){\line(1,2){1}}       \put(26,1){\line(-1,2){1}}
\put(23,-1){ (7) $P_{5}$}

\put(29,1){\circle* {0.1}}   \put(31,1){\circle* {0.1}}     \put(33,1){\circle* {0.1}}     \put(35,1){\circle* {0.1}}      \put(32,3){\circle* {0.1}}
\put(28,0.7){$v_{2}$}          \put(30,0.7){$v_{3}$}         \put(32,0.7){$v_{4}$}         \put(35,1){$v_{5}$}             \put(32,3){$v_{1}$}
\put(29,1){\line(3,2){3}}     \put(31,1){\line(1,2){1}}      \put(33,1){\line(-1,2){1}}       \put(35,1){\line(-3,2){3}}
\put(31,-1){ (8) $S_{5}$}

\put(5,7){\circle* {0.1}}   \put(7,7){\circle* {0.1}}     \put(5,9){\circle* {0.1}}     \put(7,9){\circle* {0.1}}
\put(4,6.7){$v_{3}$}          \put(7.3,6.7){$v_{4}$}       \put(4,8.7){$v_{1}$}          \put(7.3,8.7){$v_{2}$}
\put(5,7){\line(0,2){2}}    \put(5,7){\line(1,1){2}}      \put(7,7){\line(-1,1){2}}      \put(7,7){\line(0,2){2}}
\put(4,5){ (1) $K_{2,2}$}

\put(14,7){\circle* {0.1}}   \put(16,7){\circle* {0.1}}     \put(14,9){\circle* {0.1}}     \put(16,9){\circle* {0.1}}
\put(13,6.7){$v_{3}$}          \put(16.3,6.7){$v_{4}$}       \put(13,8.7){$v_{1}$}          \put(16.3,8.7){$v_{2}$}
\put(14,7){\line(0,2){2}}    \put(14,7){\line(1,1){2}}      \put(16,7){\line(-1,1){2}}
\put(14,5){ (2) $P_{4}$}

\put(23,7){\circle* {0.1}}   \put(25,7){\circle* {0.1}}     \put(27,7){\circle* {0.1}}     \put(25,9){\circle* {0.1}}
\put(22,6.7){$v_{2}$}          \put(24,6.7){$v_{3}$}       \put(26,6,7){$v_{4}$}          \put(25.3,8.7){$v_{1}$}
\put(23,7){\line(1,1){2}}    \put(25,7){\line(0,2){2}}      \put(27,7){\line(-1,1){2}}
\put(24,5){ (3) $S_{4}$}
\end{picture}
\end{center}
\vskip 1cm

By direct calculation, we obtain the following two tables.

\begin{center}
\vskip 0.5cm \begin{tabular}{|c|c|c|c|}
  \hline
   $G$      & $q^{\mathcal{D}}$  &  $q_{min}^{\mathcal{D}}$    &      $S_{\mathcal{Q}}(G)$  \\ \hline
  $K_{2,2}$      & 8                  & 2               &6              \\ \hline
  $P_{4}$       & $10.6056$        & 2               &$8.6056$    \\ \hline
  $S_{4}$        & $9.4641$        & $2.5359$     &$6.9282$     \\ \hline
\end{tabular}

\vskip.2cm Table 4.1
\end{center}

\begin{center}
\vskip 0.2cm \begin{tabular}{|c|c|c|c|}
  \hline
   $G$      & $q^{\mathcal{D}}$  &  $q_{min}^{\mathcal{D}}$    &      $S_{\mathcal{Q}}(G)$  \\ \hline
  $K_{2,3}$        & $11.3723$        & 3         &$8.3723$  \\ \hline
 $H_{1}$        & $13.3441$        & 3.3113         &$10.0328$  \\ \hline
  $H_{2}$        & $15.3119$        & 3.6075         &$11.7044$  \\ \hline
  $P_{5}$        & $17.1152$        & 3.4385         &$13.6767$  \\ \hline
  $S_{5}$        & $13.4244$        & 3.5756         &$9.8488$   \\ \hline
\end{tabular}

 \vskip.2cm Table 4.2
\end{center}

Combining Lemma \ref{lem406} and the results in Table 4.1, we get the following corollary.

\begin{cor}\label{cor4060}
For positive integers $n$ and $a$ with $2a\leq n$,
$S_{\mathcal{Q}}(K_{a,n-a})\geq S_{\mathcal{Q}}(K_{\lfloor\frac{n}{2} \rfloor,\lceil\frac{n}{2} \rceil })$
with equality if and only if $G\cong K_{\lfloor\frac{n}{2} \rfloor,\lceil\frac{n}{2} \rceil }.$
\end{cor}

Comparing the results in Tables 4.1 and 4.2, and checking more graphs with computer, it seems that among bipartite graphs, $S_{\mathcal{Q}}(K_{\lfloor\frac{n}{2} \rfloor,\lceil\frac{n}{2} \rceil })$ always has the minimum $S_{\mathcal{Q}}$. Thus, we propose the following problem for further research.
\begin{con}\label{con405}
  Let $G$ be a bipartite graph with $n$ vertices. Then
  $ S_{\mathcal{Q}}(G)\geq S_{\mathcal{Q}}(K_{\lfloor\frac{n}{2} \rfloor,\lceil\frac{n}{2} \rceil })$
with equality if and only if $G\cong K_{\lfloor\frac{n}{2} \rfloor,\lceil\frac{n}{2} \rceil }.$
\end{con}

{\bf Remark 4.2}
In order to prove Conjecture \ref{con405}, maybe it is better to show $S_{\mathcal{Q}}(G)\geq S_{\mathcal{Q}}(K_{a,n-a})$ holding for some $a$ at first, and then to using Lemma \ref{lem406} to get the desired result.

\subsection{Bound on $S_{\mathcal{Q}}$ with clique number}
\hskip 0.6cm A $clique$ of a graph $G$ is a subgraph in which any pair of vertices is adjacent,
and the clique number $\omega(G)$ (simply, $\omega$) is the number of vertices of the largest clique in $G$.
In this subsection, we present a lower bound on $S_{\mathcal{Q}}$ with clique number.

\begin{them}\label{thm408}
 Let $G$ be a simple connected graph with $n$ vertices, clique number $\omega\geq 2$ and Wiener index $W$.
Suppose that $G_{1}$, $G_{2}$, $\ldots$, $G_{k}$ are all the cliques with order $\omega$,
$s_i=\sum\limits_{v_{j}\in V(G_{i})}D_{j}$ for $ 1\leq i \leq k$. Then

\noindent {\rm(i) } if $\omega=n$, then $S_{\mathcal{Q}}(G)=n$;

\noindent {\rm(ii) } if  $2\leq \omega \leq n-1$, then
\begin{equation}\label{eq402}
S_{\mathcal{Q}}(G)\geq \max\limits_{1\leq i\leq k}{\frac{\sqrt{a_{i}^{2}-4 b_{i}(n-\omega)\omega }}{(n-\omega)\omega}},\end{equation}

\noindent where  $a_{i}=n \omega(1-\omega)+4\omega(s_i-W)-n s_i$ and $b_{i}=4W\omega(\omega-1)+4s_i(W-s_i)$.
\end{them}
\begin{proof}
(i). If $\omega=n$, then $G \cong K_{n}$. By direct calculation, we have $q^{\mathcal{D}}(G)=2n-2$ and $q_{min}^{\mathcal{D}}(G)=n-2$. Thus $S_{\mathcal{D}}(G)=q^{\mathcal{D}}(G)- q_{n}^{\mathcal{D}}(G)=n$.

(ii). If $\omega \leq n-1$, for $ 1\leq i \leq k$, suppose $V(G_{i})=\{v_{i1}$, $v_{i2}$, $\ldots$, $v_{i\omega}\}$. Then $V(G)$ is divided into two parts $V(G_{i})$ and $V(G)\setminus V(G_{i})$. Corresponding to this partition, the quotient matrix of $\mathcal{Q}(G)$ is written as

\[B= \begin{pmatrix}
   \frac{\omega(\omega-1)+s_i}{\omega} & \frac{s_i-\omega(\omega-1)}{\omega} \\
    \frac{s_i-\omega(\omega-1)}{n-\omega } & \frac{4W-3s_i+\omega(\omega-1)}{n-\omega}
  \end{pmatrix}.\]

Similar to the proof of Theorem \ref{thm302}, solving $P_{B}(\lambda)=0$ and using Lemma \ref{lem22} get (\ref{eq402}).
\end{proof}

{\bf Remark 4.3} Recall that a kite $Ki_{n,\omega}$ is the graph obtained from a clique $K_\omega$ and a path $P_{n-\omega}$ by adding an edge between an endpoint of the path and a vertex of the clique.
 For a kite $G=Ki_{5,3}$, by Theorem \ref{thm408}, we have $S_{\mathcal{Q}}(G)\geq 10.6158$.
 On the other hand, by  direct calculation, we obtain $S_{\mathcal{Q}}(G)\approx 11.3395$.
 This shows that Theorem \ref{thm408} is useful to evaluate the distance signless Laplacian spread of a graph with given clique number.




By Lemma \ref{lem22} and Theorem \ref{thm408}, we have

\begin{cor}\label{cor409}
Let $G$ be a simple connected graph with $n$ vertices and clique number $\omega$. Suppose that $G_{1}, G_{2},...,G_{k}$ are all the
cliques with order $\omega$, 
 $a_i$, $b_i$ are defined as  Theorem \ref{thm408} for $ 1\leq i \leq k$. Then

\noindent {\rm(i) } $q^{\mathcal{D}}(G)\geq \max\limits_{1\leq i\leq k}\{\frac{-a_{i}+\sqrt{a_{i}^{2}-4 b_{i}(n-\omega)\omega }}{2(n-\omega)\omega }\}$;

\noindent {\rm(ii) }  $q^{\mathcal{D}}_{min}(G) \leq \min\limits_{1\leq i\leq k}\{\frac{-a_{i}-\sqrt{a_{i}^{2}-4 b_{i}(n-\omega)\omega }}{2(n-\omega)\omega }\}$.
\end{cor}

\subsection{Bound on $S_{\mathcal{Q}}$ with diameter}

\hskip 0.6cm In this subsection, we obtain a lower  bound on $S_{\mathcal{Q}}$ of a graph with diameter. In a graph, a path is called a diameter path if its length is equal to the diameter of this graph.

\begin{them}\label{thm411}
 Let $G$ be a simple connected graph with $n$ vertices, diameter $d$ and
Wiener index $W$.
Suppose that $P_{1}$, $P_{2}$, $\ldots$, $P_{k}$ are all the diameter paths,
and suppose that $s_i=\sum\limits_{v_{j}\in V(P_{i})}D_{j}$ for $ 1\leq i \leq k$. Then

\noindent {\rm (i)} if $d=1$, then $S_{\mathcal{Q}}(G)=n$;

\noindent {\rm (ii)} if $2\leq d$, then
\begin{equation}\label{eq403}
S_{\mathcal{Q}}(G)\geq \max\limits_{1\leq i\leq k}{\frac{\sqrt{a_{i}^{2}-12 b_{i}(d+1)(n-1-d)}}{3(d+1)(n-1-d)}}, \end{equation}

\noindent where $a_{i}=12(1+d)(s_i-W)-nd(d+1)(d+2)-3ns_i$ and $b_{i}=4d(d+1)(d+2)W+12s_i(W-s_i)$.
\end{them}
\begin{proof}
(i). If $d=1$, then $G\cong K_{n}$. By direct calculation, $q^{\mathcal{D}}(G)=2n-2, q_{min}^{\mathcal{D}}(G)=n-2$. Thus,  $S_{\mathcal{D}}(G)=n$.

(ii). If $2\leq d$, for $ 1\leq i \leq k$, we let
$T=\sum \limits_{v_{s},v_{j}\in V(P_{i})} d_{sj}$. Then when $d$ is even, we have

$T=2(1+2+...+d)+2[1+1+2+3+...+(d-1)]+\ldots$

             \hskip0.4cm $+2[1+1+2+2+...+(\frac{d}{2}-1)+(\frac{d}{2}-1)+\frac{d}{2}+(\frac{d}{2}+1)]$

             \hskip0.4cm $+2[1+2+...+(\frac{d}{2}-1)+\frac{d}{2}],$

\hskip0.4cm $=d(d+1)+[(d-1)d+2\times 1]+[(d-2)(d-1)+2\times 3]+\ldots$

\hskip0.4cm $+[(\frac{d}{2}+1)(\frac{d}{2}+2)+(\frac{d}{2}-1)\frac{d}{2}]+\frac{d}{2}(\frac{d}{2}+1)$

\hskip0.4cm $=1^{2}+2^{2}+3^{2}+...+d^{2}+1+2+3+...+d$

\hskip0.4cm $=\frac{d(d+1)(d+2)}{3}.$

When $d$ is odd, we have

$T=2(1+2+...+d)+2[1+1+2+3+...+(d-1)]+\ldots$

\hskip0.4cm $+2(1+1+2+2+...+\frac{d-1}{2}+\frac{d-1}{2}+\frac{d+1}{2})$

\hskip0.4cm $=d(d+1)+[(d-1)d+2\times 1]+[(d-2)(d-1)+2\times 3]+\ldots$

\hskip0.4cm $+(\frac{d+1}{2})(\frac{d+1}{2}+1)+(\frac{d-1}{2})(\frac{d-1}{2}+1)$

\hskip0.4cm $=1^{2}+2^{2}+3^{2}+...+d^{2}+1+2+3+...+d$

\hskip0.4cm $=\frac{d(d+1)(d+2)}{3}.$

Now $V(G)$ is partitioned into  two parts which are $V(P_{i})$ and $V(G)\setminus V(P_{i})$.
Corresponding to this partition, the quotient matrix of $\mathcal{Q}(G)$ can be written as
 \[B= \begin{pmatrix}
   \frac{\frac{1}{3}d(d+1)(d+2)+s_i}{d+1} & \frac{s_i-\frac{1}{3}d(d+1)(d+2)}{d+1} \\
   \frac{s_i-\frac{1}{3}d(d+1)(d+2)}{n-d-1} & \frac{4W-3s_i+\frac{1}{3}d(d+1)(d+2)}{n-d-1}
\end{pmatrix}.\]

Similar to the proof of Theorem \ref{thm302}, solving $P_{B}(\lambda)=0$ and using Lemma \ref{lem22} get (\ref{eq403}).
\end{proof}

{\bf Remark 4.4}
For $G_{1}$ shown in Fig. 3.1, then by Theorem \ref{thm411}, we have $S_{\mathcal{Q}}(G)\geq 12.1198$. From the Table 4.1, we know that $S_{\mathcal{Q}}(G_{1})\approx 18.6100$.
This shows that Theorem \ref{thm411} is useful to evaluate the distance signless Laplacian spread
of a graph with given diameter.

\begin{cor}
Let $G$ be a simple connected graph with $n$ vertices and diameter $d$.
Suppose that  the path $P_{1}$, $P_{2}$, $\ldots$, $P_{k}$ are all the diameter of $G$,
$a_i$, $b_i$ are defined as Theorem \ref{thm411} for $ 1\leq i \leq k$. Then

\noindent {\rm(i)} $q^{\mathcal{D}}(G)\geq \max\limits_{1\leq i\leq k}\{\frac{-a_{i}+\sqrt{a_{i}^{2}-12 b_{i}(d+1)(n-1-d)}}{6(d+1)(n-1-d)}\}$;

\noindent {\rm(ii)} $q^{\mathcal{D}}_{min}(G) \leq \min\limits_{1\leq i\leq k}\{\frac{-a_{i}-\sqrt{a_{i}^{2}-12 b_{i}(d+1)(n-1-d)}}{6(d+1)(n-1-d)}\}$.
\end{cor}

\subsection{Bound On $S_{\mathcal{Q}}$ for cacti with given circumference}
\hskip 0.6cm A connected graph $G$ is a $cactus$ if any two of its cycles have at most one common vertex.
$Circumference$ is the length of the longest cycle of a graph.
In this section, we present a lower bound on $S_{\mathcal{Q}}$ of a cactus with given circumference.

\begin{them}\label{thm415}
 Let $G$ be a cactus on $n$ vertices with  circumference  $l\ (l\geq 3)$ and
Wiener index $W$.
Suppose that  cycles $C_{1}, C_{2},...,C_{k}$ are all with length $l$,
$s_i=\sum\limits_{v_{j}\in V(C_{i})}D_{j}$ for $ 1\leq i \leq k$. Then



\begin{equation}\label{eq404}
S_{\mathcal{Q}}(G)\geq \max\limits_{1\leq i\leq k}{\frac{\sqrt{a_{i}^{2}-16 b_{i}l(n-l)}}{4l(n-l)}},\end{equation}
\noindent where $$a_i=\left\{\begin{array}{cc}
                             l^{3}n+4ns_{i}-16l(s_{i}-W), & \mbox{if } l \mbox{ is even}; \\
                             l^{3}n+4ns_{i}-ln-16l(s_{i}-W), & \mbox{if } l \mbox{ is odd},
                            \end{array}\right.$$
and
$$b_i=\left\{\begin{array}{cc}
                             4l^{3}W-16s_{i}(s_{i}-W), & \mbox{if } l \mbox{ is even}; \\
                             4(l^{3}-l)W-16s_{i}(s_{i}-W), & \mbox{if } l \mbox{ is odd},
                            \end{array}\right.$$
\end{them}

\begin{proof}
Corresponding to $C_{i}$, $V(G)$ is partitioned into  two parts $V(C_{i})$ and $V(G)\setminus V(C_{i})$.

\noindent{\bf Case 1: }$l$ is even.

Then for any $v\in V(C_{i})$, the sum of distance from vertex $v$ to all other vertices on cycle $V(C_{i})$ is $\frac{l^{2}}{4}$.
 Corresponding to the above partition, the quotient matrix of $\mathcal{Q}(G)$ is written as
\[B= \begin{pmatrix}
   \frac{l^{2}}{4}+\frac{S_{i}}{l} & \frac{S_{i}}{l}-\frac{l^{2}}{4} \\
  \frac{S_{i}-\frac{l^{3}}{4}}{n-l} & \frac{4W-3S_{i}+\frac{l^{3}}{4}}{n-l}
\end{pmatrix}.\]

Similar to the proof of Theorem \ref{thm302}, solving $P_{B}(\lambda)=0$ and using Lemma \ref{lem22} get (\ref{eq402}).

\noindent{\bf Case 2: }$l$ is odd.

Then for any $v\in V(C_{i})$, the sum of distance from vertex $v$ to all other vertices on cycle $V(C_{i})$ is $\frac{l^{2}-l}{4}$.
 Corresponding to the above partition, the quotient matrix of $\mathcal{Q}(G)$ is written as
\[B= \begin{pmatrix}
   \frac{l^{2}-1}{4}+\frac{S_{i}}{l} & \frac{S_{i}}{l}-\frac{l^{2}-1}{4} \\
  \frac{S_{i}-\frac{l^{3}-l}{4}}{n-l} & \frac{4W-3S_{i}+\frac{l^{3}-l}{4}}{n-l}
\end{pmatrix}.\]

Similar to the proof of Theorem \ref{thm302}, solving $P_{B}(\lambda)=0$ and using Lemma \ref{lem22} get (\ref{eq402}).
\end{proof}
\vskip-.3cm
\setlength{\unitlength} {4mm}
\begin{center}
\begin{picture}(25,6.5)
\put(3,3){\circle* {0.1}}   \put(6,3){\circle* {0.1}}     \put(3,6){\circle* {0.1}}     \put(6,6){\circle* {0.1}}     \put(8,2){\circle* {0.1}}     \put(8,5){\circle* {0.1}}
\put(2,2.7){$v_{4}$}          \put(5.4,2.5){$v_{3}$}         \put(2,5.7){$v_{1}$}       \put(6.3,5.7){$v_{2}$}          \put(8.3,4.7){$v_{5}$}       \put(8.3,1.7){$v_{6}$}
\put(3,3){\line(0,1){3}}    \put(3,3){\line(1,0){3}}      \put(3,6){\line(1,0){3}}      \put(6,3){\line(0,1){3}}       \put(6,3){\line(1,1){2}}      \put(6,3){\line(2,-1){2}}    \put(8,2){\line(0,1){3}}
\put(4,1){ $G_3$}

\put(12,3){\circle* {0.1}}   \put(15,3){\circle* {0.1}}     \put(12,5){\circle* {0.1}}     \put(15,5){\circle* {0.1}}     \put(14,6){\circle* {0.1}}     \put(17,6){\circle* {0.1}}    \put(17,3){\circle* {0.1}}
\put(11,2.7){$v_{5}$}          \put(14.5,2){$v_{4}$}         \put(11,4.7){$v_{1}$}       \put(13,6.5){$v_{2}$}          \put(14,4.7){$v_{3}$}       \put(16,6.3){$v_{6}$}          \put(16,2.7){$v_{7}$}
\put(12,3){\line(0,1){2}}    \put(12,3){\line(1,0){3}}      \put(12,5){\line(2,1){2}}      \put(15,5){\line(0,-1){2}}       \put(15,5){\line(-1,1){1}}      \put(15,5){\line(2,1){2}}    \put(15,5){\line(1,-1){2}}      \put(17,3){\line(0,1){3}}
\put(13,1){ $G_4$} \put(6,-1){ Fig. 4.2. $G_4$, $G_4$}
\end{picture}
\end{center}
\hskip 0.1cm

{\bf Remark 4.4}
 Let $G_3$, $G_4$ are as shown in Fig. 4.2. By Theorem \ref{thm415}, we have $s_{\mathcal{Q}}(G_3)\geq 11.5$.
On the other hand, by direct calculation,  we have   $S_{\mathcal{Q}}(G_3)\approx 12.8$.
By Theorem \ref{thm415}, we have $S_{\mathcal{Q}}(G_4)\geq 13.4$.
On the other hand, by direct calculation, we have $S_{\mathcal{Q}}(G_4)\approx 16.3$.
These two examples show that Theorem \ref{thm415} is useful to evaluate the $S_{\mathcal{Q}}$ of the cacti
with given circumference.

\begin{cor}
 Let $G$ be a cactus on $n$ vertices with given circumference  $l\geq 3$.
Suppose that cycles $C_{1}, C_{2},...,C_{k}$ are all with length $l$,
$a_i,\ b_i$ are defined as Theorem \ref{thm415} for $ 1\leq i \leq k$. Then

\noindent{\rm (1)} $q^{\mathcal{D}}(G)\geq \max\limits_{1\leq i\leq k}\{\frac{-a_{i}+\sqrt{a_{i}^{2}-16 b_{i}l(n-l)}}{8l(n-l)}\}$;

\noindent{\rm (2)} $q^{\mathcal{D}}_{min}(G) \leq \min\limits_{1\leq i\leq k}\{\frac{-a_{i}-\sqrt{a_{i}^{2}-16 b_{i}l(n-l)}}{8l(n-l)}\}$.
\end{cor}


\end{document}